\documentclass[reqno, 11pt, letterpaper ]{amsart}

\usepackage{amsmath}
\usepackage{amsfonts}
\usepackage{amssymb}
\usepackage{graphicx}
\usepackage{bbm}
\usepackage{amsthm,color,yfonts,cite} \usepackage{paralist}
\usepackage{hyperref}
\usepackage{physics}
\usepackage{dsfont}

\newtheorem{theorem}{Theorem}

\newtheorem{proposition}[theorem]{Proposition}
\newtheorem{lemma}[theorem]{Lemma}

\theoremstyle{remark}
\newtheorem{remark}[theorem]{Remark}

\makeatletter
\newcommand*{\rom}[1]{\expandafter\@slowromancap\romannumeral #1@}
\makeatother

\newcommand{\ls}{\lesssim}

\newcommand{\la}{\langle}
\newcommand{\ra}{\rangle}
\newcommand{\R}{\mathbb{R}}

\newcommand{\Z}{\mathbb{Z}}
\newcommand{\pa}{\partial}
\newcommand{\ep}{\epsilon}

\usepackage{wrapfig}
\usepackage{tikz}
\usetikzlibrary{arrows,calc,decorations.pathreplacing}
\definecolor{light-gray1}{gray}{0.90}
\definecolor{light-gray2}{gray}{0.80}
\definecolor{light-gray3}{gray}{0.60}

\numberwithin{equation}{section}

\numberwithin{theorem}{section}

\numberwithin{table}{section}

\numberwithin{figure}{section}

\ifx\pdfoutput\undefined
  \DeclareGraphicsExtensions{.pstex, .eps}
\else
  \ifx\pdfoutput\relax
    \DeclareGraphicsExtensions{.pstex, .eps}
  \else
    \ifnum\pdfoutput>0
      \DeclareGraphicsExtensions{.pdf}
    \else
      \DeclareGraphicsExtensions{.pstex, .eps}
    \fi
  \fi
\fi

\title[Long-time validity of the KdV approximation]{Long-time validity of the Korteweg--de Vries approximation for the Boussinesq equation}

\date{\today}
\author[Y. Hong]{Younghun Hong}
\address{Department of Mathematics, Chung-Ang University, Seoul 06974, South Korea}
\email{yhhong@cau.ac.kr}

\author[J. Jang]{Junyeong Jang}
\address{Department of Mathematics, Chung-Ang University, Seoul 06974, South Korea}
\email{jyjang0119@cau.ac.kr}

\begin{document}

\begin{abstract}
We consider the Korteweg--de Vries (KdV) approximation for the good Boussinesq equation. Building upon the low-regularity local-in-time justification established by Hong and Yang \cite{HY2024}, we prove that this approximation remains valid on logarithmically long time intervals. The key ingredient is a persistence of regularity argument, which allows the local approximation to be iterated using the rescaled conservation law for the Boussinesq equation.
\end{abstract}

\maketitle

\section{Introduction}

\subsection{Motivation}

Modulation approximation captures the leading-order dynamics of a complicated dynamical system through a reduced equation in an appropriate asymptotic regime. Such reduced equations are often derived formally by substituting an asymptotic ansatz into the original equation and collecting the leading-order terms. These reduced equations are referred to as modulation equations. This approach has become a fundamental tool for understanding multiscale PDE models and the underlying physical phenomena through simpler effective equations \cite{SU2017, Abl2011}.

To justify a modulation approximation rigorously, one must establish quantitative error estimates between the solutions of the original equation and the corresponding modulation equation over an appropriate time interval. More precisely, one seeks to identify the class of initial data for which the approximation can be rigorously justified and the time interval over which it remains mathematically valid. Considerable progress has been made in this direction; we refer to the monograph \cite{SU2017} for the general theory and to \cite{Cra1985, Dul2012, SW2000, SW2002, Sch1998, GP2014, Han2013, SW2000-2} for representative examples. Nevertheless, existing justification results often require relatively high regularity or additional assumptions on the initial data and, in many cases, establish the approximation only on local time intervals. These limitations motivate the present work.

In this paper, we are interested in extending the validity of modulation approximations in two closely related directions. The first direction is to broaden the class of solutions for which the approximation can be rigorously justified. For example, the Korteweg--de Vries (KdV) equation is known to be a universal model, as it arises as a modulation equation for a wide variety of dispersive systems, including the full water wave problem in the shallow-water regime \cite{Cra1985, Lan2013, Dul2012, SW2000, SW2002, Sch1998}. Motivated both by the presence of rough waves in physical applications and by the development of low-regularity well-posedness theory for the KdV equation \cite{Bou1993-2, KPV1996, KPV1991, KPV1993-1, KPV1993-2, CKSTT2003-1, KV2019, KT2006}, it is natural to seek rigorous justification of the modulation approximation for rough solutions. The second direction is to extend the interval of validity by exploiting conservation laws. In many situations, however, the available conservation laws provide only low-regularity control, making them insufficient for directly iterating a local approximation established at a higher regularity level. Therefore, refining the local theory so that its lifespan is determined only by the lower norm controlled by the conservation law is essential for extending the validity to longer time scales.

Motivated by this observation, we show that once the local theory is refined so that the lifespan depends only on the lower norm controlled by the conservation law, higher Sobolev regularity can be propagated on the same interval and enters only through the local error. This allows the local approximation to be iterated and yields logarithmically long-time validity for the Boussinesq--KdV approximation.

\subsection{The Boussinesq--KdV approximation}

To illustrate the main ideas while avoiding unnecessary technical complications, we use the Boussinesq--KdV approximation as a model problem. Specifically, we consider the \emph{good} Boussinesq equation
\begin{equation}\label{eq: Boussinesq}
\pa_t^2u
=\pa_x^2u-\pa_x^4u-\pa_x^2(u^2),
\end{equation}
where $u=u(t,x):I(\R)\times\R\to\R$. This equation is variant of the classical equation introduced by Boussinesq in his study of shallow water waves \cite{Bou1872}. In the long-wave regime, its leading-order dynamics are described by two counter-propagating KdV equations.

More precisely, for $0<\ep\ll1$, we seek a long-wave solution of the form
\begin{equation}\label{eq: long wave scaling}
u(t,x)=\epsilon^2 u_\ep(\epsilon^3t,\epsilon x).
\end{equation}
Substituting \eqref{eq: long wave scaling} into \eqref{eq: Boussinesq}, we find that $u_\ep$ satisfies the rescaled Boussinesq equation
\begin{equation}\label{eq: rescaled Boussinesq}
\ep^4\pa_t^2u_\ep
=
\pa_x^2(1-\ep^2\pa_x^2)u_\ep
-
\ep^2\pa_x^2(u_\ep^2).
\end{equation}
A formal analysis (see Section~\ref{subsec: Derivation of the coupled Boussinesq system}) shows that the solution decomposes into two counter-propagating waves,
\begin{equation}\label{decomposition relation}
u_\ep(t,x)
=
u_\ep^+\big(t,x-\tfrac{t}{\ep^2}\big)
+
u_\ep^-\big(t,x+\tfrac{t}{\ep^2}\big),
\end{equation}
where the profiles $u_\ep^\pm$ satisfy the rescaled coupled Boussinesq system
\begin{equation}\label{eq: rescaled coupled Boussinesq system}
\pa_tu_{\ep}^{\pm}(t,x)
=
\pm
\frac{\pa_x^3}{\sqrt{1-\ep^2\pa_x^2}+1}
u_\ep^\pm(t,x)
\pm
\frac{1}{2}
\frac{\pa_x}{\sqrt{1-\ep^2\pa_x^2}}
(u_\ep^2)\bigg(t,x\pm\frac{t}{\ep^2}\bigg),
\end{equation}
where $\sqrt{1-\ep^2\pa_x^2}$ denotes the Fourier multiplier with symbol $\sqrt{1+\ep^2\xi^2}$. Then, in the long-wave limit $\ep\to0$, each profile $u_\ep^\pm$ formally converges to a solution of the KdV equation
\begin{equation}\label{eq: KdV}
2\pa_tw^\pm\mp\pa_x^3w^\pm\mp2w^\pm\pa_xw^\pm
=0.
\end{equation}
Thus, solutions of the rescaled Boussinesq equation \eqref{eq: rescaled Boussinesq} are formally approximated by the superposition of two decoupled KdV waves:
$$
u_\ep(t,x)
\underset{\ep\to0}{\approx}
w^+\big(t,x-\tfrac{t}{\ep^2}\big)
+
w^-\big(t,x+\tfrac{t}{\ep^2}\big).
$$
This Boussinesq--KdV approximation was rigorously justified by Schneider \cite{Sch1998} under high-regularity assumptions. More recently, Hong and Yang \cite{HY2024} refined this result by establishing a local-in-time justification for initial data with regularity only slightly above $L^2$, while simultaneously removing the weighted norm assumption.

Besides its physical relevance, this model problem is particularly well suited to illustrating our approach to extending the interval of validity in situations where the available conservation laws provide only low-regularity control. In contrast, a typical example in which conservation laws can be straightforwardly exploited is the Benjamin--Bona--Mahony (BBM) equation \cite{BBM1972, BPS1983, Olver1979}. For the BBM equation, the conserved quantities yield $\ep$-uniform $H^1$ bounds, while the regularity requirement for the rigorous justification of the BBM approximation can be reduced to $H^1$. Consequently, the local approximation can be iterated to obtain an exponential-in-time error estimate \cite{HJY2025}. This illustrates that lowering the regularity requirement to match the level of the available conservation laws is not merely a technical improvement, but is also essential for extending the interval of validity.

One of the major obstructions in the rigorous justification of long-wave approximations is that the available conservation laws often fail to provide uniform control of Sobolev norms after the long-wave scaling. This phenomenon is not specific to the Boussinesq equation. Similar difficulties also arise in several other long-wave approximation problems; for example, in the derivation of the KdV equation from the Fermi--Pasta--Ulam system \cite{HKY2021}. Nevertheless, the Boussinesq equation provides a particularly simple setting in which this phenomenon can be clearly illustrated.

Indeed, the Boussinesq equation \eqref{eq: Boussinesq} conserves the energy (see \cite{Lin1993})
\begin{equation}\label{eq: energy for Boussinesq}
\mathcal{E}_0[u]
=
\frac12\int_\R
\Big\{
u^2
+
(\partial_xu)^2
+
(\pa_x^{-1}\partial_tu)^2
\Big\}\,dx
-
\frac13\int_\R u^3\,dx.
\end{equation}
Under the long-wave scaling \eqref{eq: long wave scaling}, the rescaled solution $u_\ep$ satisfies the conservation law
\begin{equation}\label{eq: energy for rescaled Boussinesq}
\mathcal{E}_\ep[u_\ep]
=
\frac12\int_\R
\Big\{
u_\ep^2
+
\ep^2(\partial_xu_\ep)^2
+
\ep^4(\pa_x^{-1}\partial_tu_\ep)^2
\Big\}\,dx
-
\frac{\ep^2}{3}\int_\R u_\ep^3\,dx,
\end{equation}
Since the derivative terms are multiplied by powers of $\ep$, the conserved energy provides only $L^2$ control in the small-$\ep$ regime and therefore does not yield a uniform $H^1$ bound. Hence, unlike the BBM equation, the local approximation cannot be iterated directly using the conservation law alone.

For this reason, the Boussinesq--KdV approximation provides a particularly transparent setting in which to develop and illustrate a persistence-of-regularity approach for extending the interval of validity in situations where the available conservation laws provide only low-regularity control.

\subsection{Main result}

Our main theorem establishes an exponential-in-time $L^2$ error estimate for the Boussinesq--KdV approximation in a low-regularity setting by combining a persistence of regularity argument with the rescaled conservation law. As a consequence, the interval of validity extends to logarithmically long time scales.

\begin{theorem}[Exponential-in-time KdV approximation for the Boussinesq equation]\label{thm: main theorem}
Let $0<s\le5$. Suppose that
$$
\sup_{0<\ep\leq 1}\sum_\pm
\big\|\la\ep\pa_x\ra u_{\ep,0}^\pm\big\|_{H_x^s}
<\infty\quad\textup{and}\quad \sum_\pm\|w_0^\pm\|_{H_x^s}<\infty,
$$
and let $u_\ep^\pm(t)\in C(\R;H_x^s)$ (resp. $w^\pm(t)\in C(\R;H_x^s)$) denote the unique solution to the rescaled coupled Boussinesq system \eqref{eq: rescaled coupled Boussinesq system} (resp. the KdV equation \eqref{eq: KdV}) with initial data $u_{\ep,0}^\pm$ (resp. $w_0^\pm$). Then there exist $\ep_0\in(0,1]$ and constants $K_0,K_s>0$ such that for every $\ep\in(0,\ep_0]$ and $t\in\mathbb{R}$,
\begin{equation}\label{eq: convergence bound}
\|u_\ep^\pm(t)-w^\pm(t)\|_{L_x^2}
\leq
e^{K_0|t|}
\left(K_0\|u_{\ep,0}^\pm-w_0^\pm\|_{L_x^2}
+
K_s\ep^{\min\{\frac{2s}{5},\frac{1}{2}\}}\right).
\end{equation}
\end{theorem}

\begin{remark}[Dependence of the constants]
The constants $K_0$ and $K_s$ in Theorem~\ref{thm: main theorem} are independent of $\ep$ and $t$. For fixed $s$, the exponential rate $K_0$ is determined only by the lower-regularity bounds 
$$
\sup_{0<\ep\leq 1}\sum_\pm\|\la\ep\pa_x\ra u_{\ep,0}^\pm\|_{L_x^2}
\quad
\textup{and}
\quad
\sum_\pm\|w_0^\pm\|_{L_x^2},
$$
whereas $K_s$ may additionally depend on the $H^s$-bounds in the hypothesis. The key point for the iteration is that the local lifespan depends only on these lower-regularity bounds and not on the higher Sobolev bounds. The role of this distinction in the iteration argument is explained in Section~\ref{subsubsec: persistence of regularity} below.
\end{remark}

\begin{remark}[Logarithmic time validity]
Theorem~\ref{thm: main theorem} implies that the KdV approximation remains valid on time intervals satisfying
$$
|t|
\ll
\frac{1}{K_0}
\log\left(\frac{1}{K_0\|u_{\ep,0}^\pm-w_0^\pm\|_{L_x^2}+K_s\ep^{\min\{\frac{2s}{5},\frac{1}{2}\}}}\right).
$$
In particular, if
$$
\|u_{\ep,0}^\pm-w_0^\pm\|_{L_x^2}
\ls
\ep^{\min\{\frac{2s}{5},\frac{1}{2}\}},
$$
for example, when $u_{\ep,0}^\pm=w_0^\pm$, then
$$
\lim_{\ep\to0}
\sup_{|t|\le c\log(1/\ep)}
\|u_\ep^\pm(t)-w^\pm(t)\|_{L_x^2}=0,
$$
for every
$0<c< \frac{1}{K_0} \min\{\frac{2s}{5},\frac{1}{2}\}$. Thus, in the rescaled time variable, the interval on which the KdV approximation remains valid grows at least logarithmically as $\ep\to0$, improving the local-in-time justification in Hong--Yang \cite{HY2024}.
\end{remark}

\begin{remark}
The argument developed in this paper provides a robust framework for modulation approximation problems in which the available conservation laws give only lower-regularity control and therefore do not directly allow one to iterate a local approximation established at a higher regularity level. The key step is to refine the local theory so that the lifespan is determined solely by the lower norm controlled by the conservation law. Higher Sobolev regularity can then be propagated on the same time interval and enters only through the size of the local error. This makes it possible to restart the local approximation uniformly in time and thereby extend the interval of validity. The Boussinesq--KdV approximation serves as a simple model problem illustrating this mechanism.
\end{remark}

Finally, scaling back, we obtain the following formulation in the original variables.
\begin{remark}[Long-wave approximation for the Boussinesq equation]\label{rmk: reformulation of main theorem}
Let $0<s\leq 5$, and suppose that
$$
\sum_\pm
\sup_{0<\ep\leq 1}
\|\la\ep\pa_x\ra u_{\ep,0}^\pm\|_{H_x^s}
<\infty.
$$
For each $\ep\in (0,1]$, let $U_\ep(t)\in C(\R;H_x^s)$ be the solution to the Boussinesq equation \eqref{eq: Boussinesq} with initial data
$$
\Big(U_\ep(0,x), \, \,  \pa_tU_\ep(0,x)\Big)
=
\Big(\ep^2(u_{\ep,0}^++u_{\ep,0}^-)(\ep x),\, \,
\ep^2\pa_x\sqrt{1-\pa_x^2}\big
((u_{\ep,0}^--u_{\ep,0}^+)(\ep x)\big)\Big),
$$
and $W_\ep^\pm(t)\in C(\R;H_x^s)$ denote the solution to the KdV equation \eqref{eq: KdV} with initial data $u_{\ep,0}^\pm$. Then there exist $\ep_0\in(0,1]$ and $K_0,K_s>0$, such that for every $\ep\in (0,\ep_0]$ and $T>0$,
$$
\sup_{|t|\leq T/\ep^3}
\bigg\|U_\ep(t,\, x)
-\ep^2\sum_\pm W_\ep^\pm\big(\ep^3t, \, \ep(x\mp t)\big)
\bigg\|_{L_x^2}
\leq
K_s
\ep^{\frac{3}{2}+\min\{\frac{2s}{5},\frac{1}{2}\}}
e^{K_0T}.
$$
\end{remark}

\subsection{Ideas of the proof}
The proof combines the Fourier analytic approach of Hong and Yang \cite{HY2024} with a persistence of regularity argument.

\subsubsection{Fourier analytic approach}

For dispersive equations, modulation approximations can often be justified under low-regularity assumptions by exploiting the Fourier analytic techniques developed by Kenig--Ponce--Vega \cite{KPV1996}, together with the Bourgain space framework introduced in \cite{Bou1993-1, Bou1993-2}. This approach was first applied to the KdV approximation for the Fermi--Pasta--Ulam system by Hong--Kwak--Yang \cite{HKY2021}, where the regularity requirement in the earlier justification result of Schneider--Wayne \cite{SW2000-2} was lowered. Similarly, Fourier analytic techniques have also been successfully applied to several approximation problems, including continuum limits for discrete equations and related modulation approximation problems; see, for example, \cite{HJY2025, HKY2021, HY2024, HKNY2021, HKY2023, HY2019, KY2025}.

\subsubsection{Persistence of regularity}\label{subsubsec: persistence of regularity}

For the KdV approximation of the Boussinesq equation, Hong and Yang \cite{HY2024} employed this Fourier analytic approach to lower the regularity requirement. However, as discussed above, unlike the BBM case, the resulting local approximation cannot be iterated directly using the conservation law alone.

In the present paper, we overcome this obstruction by refining the local approximation in \cite{HY2024} through a persistence of regularity argument. The key point is that the local lifespan depends only on the lower $L^2$ norm, while higher $H^s$ regularity is propagated on the same interval. Since the rescaled conservation law provides uniform $L^2$ control, the refined local approximation can therefore be iterated with a uniform time step.

The argument relies on two simple observations. First, the nonlinear estimates can be refined so that the higher-order derivatives are placed on a single factor (see Lemma~\ref{lem: bilinear estimates} and Remark~\ref{remark: refinement of bilinear estimates}). This follows from the elementary inequality
\begin{equation}\label{ineq: fractional Leibniz rule}
\bigg\langle\sum_{j=1}^m\xi_j\bigg\rangle^s
\leq
C_{s,m}
\sum_{l=1}^m
\langle\xi_l\rangle^s
\end{equation}
for $s\geq0$. Second, the closed ball, defined by the higher Sobolev norm but equipped with the lower-regularity metric, is complete (see Lemma~\ref{lem: Completeness of the fixed point space} and \eqref{eq: contraction mapping space}). This allows the Banach fixed point theorem to be applied at the lower-regularity level while simultaneously propagating higher Sobolev regularity. 

Although the higher Sobolev norms may grow during the iteration, Proposition~\ref{prop: Hs norm bound} shows that this growth is at most exponential-in-time. Thus, the local approximation can be iterated to obtain the exponential-in-time error estimate stated in Theorem~\ref{thm: main theorem}. We have also employed this strategy in related settings \cite{HJ2026-1, HJ2026-2}.

\subsection{Outline of the paper}

The rest of the paper is organized as follows. In Section~\ref{sec: Preliminaries}, we derive the coupled Boussinesq system, introduce the associated Bourgain spaces, and review their basic properties. In Section~\ref{sec: Uniform bounds for the Boussinesq system}, we refine the bilinear estimates to establish local uniform bounds with persistence of regularity and derive exponential-in-time $H^s$ bounds from the uniform $L^2$ control. Finally, in Section~\ref{sec: KdV Approximation and Global Extension}, we refine the local KdV approximation via persistence of regularity and iterate the resulting estimate to prove the main theorem.

\subsection{Notation}
Throughout the paper, $\ep\in (0,1)$ denotes a small parameter. For two nonnegative quantities $A$ and $B$ we denote $A\ls B$ if $A\leq CB$ for some constant $C>0$, independent of $\ep$, and $A\sim B$ if $A\ls B$ and $B\ls A$. The Japanese bracket is defined by
$$
\la x\ra:=\sqrt{1+|x|^2}.
$$
Let $\hat{f}$ denote the Fourier transform of $f$. We write $|\pa_x|$ and $\la\ep\pa_x\ra$ for the Fourier multipliers with symbols $|\xi|$ and $\la\ep\xi\ra$, respectively, and denote by $e^{a\pa_x}$ the translation operator,
$$
e^{a\pa_x}f(x)=f(x+a).
$$
For $N>0$, we denote by $P_{\le N}$ the sharp frequency projection defined by
$$
\widehat{P_{\leq N}f}(\xi)
=
\mathbf{1}_{\{|\xi|\leq N\}}
\hat{f}(\xi).
$$
Finally, for $T>0$, we define the smooth time cut-off $\eta_T(t)$ by
\begin{equation}\label{eq: smooth cutoff}
\eta_T(t):=\eta(t/T)
\quad \textup{where}\quad
0\leq\eta\leq1, \quad \eta(t)\equiv 1 \,\textup{ on }[-1,1],\quad\operatorname{supp}\eta\subset[-2,2].
\end{equation}

\subsection{Acknowledgements}
This work was supported by the National Research Foundation of Korea (NRF) grant funded by the Korean government (MSIT) (No. RS-2023-00219980 and RS-2026-25479401).

\section{Preliminaries}\label{sec: Preliminaries}

In this section, we develop the analytical framework for the low-regularity justification of the KdV approximation. In Section~\ref{subsec: Derivation of the coupled Boussinesq system}, starting from the long-wave ansatz \eqref{eq: long wave ansatz for Boussinesq}, we derive the coupled Boussinesq system in integral form \eqref{eq: coupled Boussinesq integral form}, which formally converges to the two decoupled KdV equations \eqref{eq: KdV}. This integral formulation is essential for the low-regularity analysis, since a direct comparison at the level of the differential equations would require higher regularity, particularly for the convergence of the linear parts \cite[Chapter~12]{SU2017}. In Section~\ref{subsec: Bourgain spaces}, we introduce the Bourgain spaces associated with the linear Boussinesq flow and recall the estimates that will be used throughout the paper.

\subsection{Derivation of the coupled Boussinesq system}\label{subsec: Derivation of the coupled Boussinesq system}

For completeness, we briefly recall the derivation of the coupled Boussinesq system \eqref{eq: coupled Boussinesq integral form} following \cite{HY2024}. Starting from the long-wave ansatz, we derive an equivalent integral formulation by rescaling the equation, diagonalizing the linear part into two counter-propagating modes, and factoring out the leading-order transport. This integral formulation is better suited to the low-regularity analysis and will serve as the starting point for the Bourgain space estimates developed in the remainder of the paper.

Starting from the long-wave ansatz
\begin{equation}\label{eq: long wave ansatz for Boussinesq}
u(t,x)
=
\ep^2u_\ep^+(\ep^3t,\ep(x-t))
+
\ep^2u_\ep^-(\ep^3t,\ep(x+t)),
\end{equation}
we define
\begin{equation}\label{eq: rescaled u}
u_\ep(t,x):=\frac{1}{\ep^2}
u\bigg(\frac{t}{\ep^3}, \frac{x}{\ep}\bigg).
\end{equation}
Substituting \eqref{eq: rescaled u} into \eqref{eq: Boussinesq} yields the rescaled equation
\begin{equation}\label{eq: rescaled equation}
\left\{\begin{aligned}
\ep^4\pa_t^2u_\ep
&=
\pa_x^2\la\ep\pa_x\ra^2u_\ep
-\ep^2\pa_x^2(u_\ep^2),\\
u_\epsilon(0)&=u_{\epsilon,0},\\
\partial_tu_\epsilon(0)&=u_{\epsilon,1}.
\end{aligned}\right.
\end{equation}
To rewrite the equation as a first-order system, we introduce
$$
\mathbf{y}:=
\begin{bmatrix}
u_\ep
\\
\frac{\ep^2}{\pa_x\la\ep\pa_x\ra}\pa_t u_\ep
\end{bmatrix},
\qquad
\mathbf{M}:=
\frac{1}{\ep^2}\pa_x\la\ep\pa_x\ra
\begin{bmatrix}
0 & 1\\
1 & 0
\end{bmatrix},
\qquad
\mathbf{N}:=
-\frac{\pa_x}{\la\ep\pa_x\ra}(u_\ep^2)
\begin{bmatrix}
0\\
1
\end{bmatrix},
$$
so that \eqref{eq: rescaled equation} becomes
\begin{equation}\label{eq: first order system}
\pa_t\mathbf{y}
=
\mathbf{M}\mathbf{y}+\mathbf{N}.
\end{equation}
To diagonalize the linear part, we define
\begin{equation}\label{eq: transformed variable}
\begin{bmatrix}
z_\ep^-\\
z_\ep^+
\end{bmatrix}
:=
\mathbf{U}\mathbf{y}
=
\frac{1}{\sqrt{2}}
\begin{bmatrix}
u_\ep + \frac{\ep^2}{\pa_x\la\ep\pa_x\ra}\pa_t u_\ep
\\
u_\ep - \frac{\ep^2}{\pa_x\la\ep\pa_x\ra}\pa_t u_\ep
\end{bmatrix},
\qquad
\mathbf{U}=
\frac{1}{\sqrt{2}}
\begin{bmatrix}
1 & 1\\
1 & -1
\end{bmatrix}.
\end{equation}
Since $\mathbf{U}^{-1}=\mathbf{U}$, the linear operator is diagonalized as
$$
\partial_t
\begin{bmatrix}
z_\ep^-\\
z_\ep^+
\end{bmatrix}
=
\frac{1}{\ep^2}
\pa_x\la\ep\pa_x\ra
\begin{bmatrix}
1&0\\
0&-1
\end{bmatrix}
\begin{bmatrix}
z_\ep^-\\
z_\ep^+
\end{bmatrix}
+
\frac{1}{\sqrt2}
\frac{\pa_x}{\la\ep\pa_x\ra}(u_\ep^2)
\begin{bmatrix}
-1\\1
\end{bmatrix},
$$
or equivalently,
\begin{equation}\label{eq: diagonalized system}
\pa_tz_\ep^\pm
=
\mp\frac{1}{\ep^2}
\pa_x\la\ep\pa_x\ra z_\ep^\pm
\pm
\frac{1}{\sqrt2}
\frac{\pa_x}{\la\ep\pa_x\ra}(u_\ep^2).
\end{equation}
Motivated by the long-wave ansatz \eqref{eq: long wave ansatz for Boussinesq}, we introduce
\begin{equation}\label{eq: definition of u epsilon}
u_\ep^\pm(t,x)
:=
\frac1{\sqrt2}
z_\ep^\pm\left(t,x\pm\frac{t}{\ep^2}\right),
\end{equation}
which yields the counter-propagating decomposition \eqref{decomposition relation}. This moving frame removes the dominant transport and isolates the dispersive dynamics. Substituting \eqref{eq: definition of u epsilon} into \eqref{eq: diagonalized system}, we obtain
\begin{equation}\label{eq: coupled Boussinesq}
\begin{aligned}
\pa_tu_{\ep}^{\pm}(t,x)
&=\mp\frac{\pa_x}{\ep^2}\big(\la\ep\pa_x\ra-1\big)u_\ep^\pm
\pm\frac12
\frac{\pa_x}{\la\ep\pa_x\ra}
(u_\ep^2)\left(t,x\pm\frac{t}{\ep^2}\right)\\
&=
\pm
\frac{\pa_x^3}{\la\ep\pa_x\ra+1}
u_\ep^\pm
\pm
\frac12
\frac{\pa_x}{\la\ep\pa_x\ra}
(u_\ep^2)\left(t,x\pm\frac{t}{\ep^2}\right).
\end{aligned}
\end{equation}

To work in a low-regularity setting, we rewrite the coupled Boussinesq system \eqref{eq: coupled Boussinesq} in integral form. To this end, we introduce the rescaled Boussinesq phase function
\begin{equation}\label{eq: Boussinesq phase function}
s_\ep (\xi)
:=\frac{\xi}{\ep^2}(\la\ep\xi\ra-1)
=\frac{\xi^3}{\la\ep\xi\ra+1},
\end{equation}
and the associated linear Boussinesq propagators
\begin{equation}\label{eq: coupled Boussinesq flow}
S_\ep^{\pm}(t)
:=e^{\mp it s_\ep(-i\pa_x)}
=
e^{\mp\frac{t}{\ep^2}\pa_x(\la\ep\pa_x\ra-1)}.
\end{equation}
Here, $e^{a\pa_x}$ denotes the translation operator in the spatial variable, namely,
$$
e^{a\pa_x}f(x)=f(x+a).
$$
Applying Duhamel's formula to \eqref{eq: coupled Boussinesq}, together with the decomposition \eqref{decomposition relation} and the identity
$$
e^{a\pa_x}(f^2)=f(x+a)^2=(e^{a\pa_x}f)^2,
$$
we obtain the equivalent integral formulation
\begin{equation}\label{eq: coupled Boussinesq integral form}
\boxed{\quad u_{\ep}^{\pm}(t)
=
S_\ep^{\pm}(t) u_{\ep,0}^{\pm}
\pm \frac{1}{2}\int_0^t
S_\ep^{\pm}(t-t_1)\frac{\pa_x}{\la\ep\pa_x\ra}
\Big\{
\big(u_\ep^\pm (t_1)+ e^{\pm \frac{2t_1}{\ep^2}\pa_x} u_\ep^\mp (t_1) \big)^2
\Big\}
\, dt_1.\quad}
\end{equation}
The initial data are given by
\begin{equation}\label{initial data for coupled Boussinesq}
u_\ep^\pm(0)
:=
u_{\ep,0}^{\pm}
=\frac{1}{2}
\bigg(
u_\ep(0)\mp 
\frac{\ep^2}{\la\ep\pa_x\ra}\pa_x^{-1}\pa_tu_\ep(0)
\bigg).
\end{equation}
The integral equation \eqref{eq: coupled Boussinesq integral form} will serve as the starting point for all subsequent estimates.

\begin{remark}[Formal convergence in integral form]
The integral formulation \eqref{eq: coupled Boussinesq integral form} makes the connection with the KdV approximation transparent. The integral equation \eqref{eq: coupled Boussinesq integral form} formally converges, as $\ep\to0$, to the KdV equation in integral form
\begin{equation}\label{eq: intergal KdV}
\boxed{\quad w^\pm(t)
=
S^\pm(t)w_{0}^\pm
\pm
\frac{1}{2}
\int_0^t S^\pm(t-t_1)\pa_x\big(w^\pm(t_1)\big)^2 \, dt_1,\quad}
\end{equation}
where $S^\pm(t)$ denotes the Airy flow
\begin{equation}\label{eq: Airy flow}
S^{\pm}(t)
:=e^{\mp it s(-i\pa_x)}
=
e^{\pm\frac{t}{2}\pa_x^3},
\qquad
s(\xi)=\frac{\xi^3}{2},
\end{equation}
since the phase function $s_\ep(\xi)$ converges to $s(\xi)$ as $\ep\to0$; see \eqref{eq: Boussinesq phase function}. In view of the decomposition \eqref{decomposition relation}, this is precisely the formal derivation of the two decoupled KdV equations \eqref{eq: KdV} from the Boussinesq equation \eqref{eq: Boussinesq} under the long-wave ansatz \eqref{eq: long wave ansatz for Boussinesq}.
\end{remark}

\subsection{Bourgain spaces}\label{subsec: Bourgain spaces}

We now introduce the Bourgain spaces (also known as $X^{s,b}$--spaces or Fourier restriction spaces) which will be used to capture the dispersive smoothing effect. For $s,b\in\R$, we define the Bourgain norm associated with the rescaled Boussinesq flow $S_\ep^\pm(t)$ (resp. Airy flow $S^\pm(t)$) by
$$
\|u\|_{X_{\ep,\pm}^{s,b}}
:=\|\la\xi\ra^s\la\tau\pm s_\ep(\xi)\ra^b\tilde{u}\|_{L_{\tau, \xi}^2(\R\times\R)}
\quad
(resp. \ \ \|u\|_{X_{\pm}^{s,b}}
:=\|\la\xi\ra^s\la\tau\pm s(\xi)\ra^b\tilde{u}\|_{L_{\tau, \xi}^2(\R\times\R)}).
$$
Here, $\tilde{u}(\tau,\xi)$ denotes the space-time Fourier transform of $u(t,x)$ given by
$$
\tilde{u}(\tau,\xi)
:=
\int_\R \int_\R
u(t,x) e^{-i(t\tau+x\xi)}\, dx \,dt.
$$

\begin{remark}[Sensitivity of the Bourgain norm]\label{rmk: Sensitivity of the Bourgain norm}

The Bourgain norm is adapted to the underlying phase function and captures the smoothing effect of dispersion. Thus, directly comparing $u_\ep^\pm$ and $w^\pm$ in a single Bourgain norm may require additional regularity, since they are associated with different phase functions. To overcome this difficulty, in Section~\ref{subsec: Frequency localized decomposition} below, we will analyze this issue and introduce a frequency-localized auxiliary equation as in \cite{HY2024}.
\end{remark}

We recall basic estimates for the general Bourgain space. These estimates will be used in the local theory and in the approximation estimates below. We refer the reader to \cite{Tao2006, LP2015} for detailed proofs.
\begin{lemma}[Basic estimates for the Bourgain norm]\label{lem: properties of Bourgain norm}
Let $T\in(0,1]$ and set $\eta_T(t)=\eta(t/T)$ where $\eta\in C_c^\infty(\R)$ is the smooth time cut-off defined in \eqref{eq: smooth cutoff}. Let $X^{s,b}$ denote either $X_{\ep,\pm}^{s,b}$ or $X_\pm^{s,b}$, and $S(t)$ denote the corresponding linear flow.
Then, for any $s,b\in\R$, the following estimates hold.
\begin{enumerate}[$(i)$]
\item (Embedding) If $b>\frac{1}{2}$, then $\|u\|_{C_t(\R;H_x^s)}\ls \|u\|_{X^{s,b}}$.
\item (Linear flow estimate) 
If $b>\frac{1}{2}$, then
$$\|\eta_T(t)S(t) u_0\|_{X^{s,b}} \lesssim T^{\frac{1}{2}-b}\|u_0\|_{H^s_x}.$$
\item (Stability with respect to smooth time cut-off)
If $-\frac{1}{2}<b'\leq b<\frac{1}{2}$, then
$$
\|\eta_T(t)u\|_{X^{s,b'}}
\ls
T^{b-b'}\|u\|_{X^{s,b}}.
$$
Moreover, if $\frac{1}{2}<b\leq1$, then
$$
\|\eta_T(t)u\|_{X^{s,b}}
\ls
T^{\frac{1}{2}-b}\|u\|_{X^{s,b}}.
$$

\item (Inhomogeneous term estimate) 
If $\frac{1}{2}<b\leq 1$, then
$$
\bigg\|
\eta_T(t) \int_0^tS(t-t_1)F(t_1)dt_1
\bigg
\|_{X^{s,b}} \ls T^{\frac{1}{2}-b}\|F\|_{X^{s,b-1}}.
$$
\end{enumerate}
\end{lemma}

In the next section, we establish uniform bounds for solutions to the Boussinesq system via a standard contraction argument in Bourgain spaces. To incorporate persistence of regularity, which is the key new ingredient of this paper, we introduce a complete metric space equipped with a weaker metric. This construction is standard in the study of high Sobolev norm bounds for nonlinear dispersive equations, and we adapt it to the present setting.

\begin{lemma}[Completeness of the fixed point space]\label{lem: Completeness of the fixed point space}
Let $s\geq 0$, $b\in\R$, and $M_0, M_s \geq 0$. We define
\begin{equation}\label{eq: fixed point space}
\mathcal{X}:=\Big\{ u \in X_{\ep,\pm}^{s,b}:\|u\|_{X_{\ep,\pm}^{s,b}}\leq M_s,\quad \|u\|_{X_{\ep,\pm}^{0,b}}\leq M_0\Big\},
\end{equation}
equipped with the metric
$$
d(u, v):=\|u-v\|_{X_{\ep,\pm}^{0,b}}.
$$
Then $(\mathcal{X},d)$ is a complete metric space.
The same statement holds when $X_{\ep,\pm}^{s,b}$ is replaced by $X_{\pm}^{s,b}$.
\end{lemma}

Lemma~\ref{lem: Completeness of the fixed point space} is a direct consequence of the following general result; see \cite[Theorem~1.2.5]{Caz2003} and \cite[Proof of Proposition~3.2]{Soh2011-1}.

\begin{proposition}
Let $X\hookrightarrow Y$ be two Banach spaces, and let $1<p,q\leq \infty$. Let $I\subseteq\R$ be an open interval, possibly $I=\R$. Suppose that $\{f_n\}_{n\geq 0}$ is bounded in $L^q(I;Y)$, and let $f:I\to Y$ be such that $f_n(t)\rightharpoonup f(t)$ in $Y$ as $n\to\infty$, for almost every $t\in I$. If $\{f_n\}_{n\geq 0}$ is bounded in $L^p(I;X)$ and $X$ is reflexive, then $f\in L^p(I;X)$ and
$$
\|f\|_{L^p(I;X)}\leq\liminf_{n\to\infty}\|f_n\|_{L^p(I;X)}.
$$
\end{proposition}

This criterion has been used in the local $H^1$ well-posedness theory for the nonlinear Schrödinger equation to establish the completeness of fixed point spaces defined by higher Sobolev norms and equipped with weaker metrics; see \cite[Proof of Theorem~4.4.1]{Caz2003}.

\section{Uniform bounds for the Boussinesq system}\label{sec: Uniform bounds for the Boussinesq system}

In this section, we study the initial value problem for the rescaled Boussinesq system in the integral formulation \eqref{eq: coupled Boussinesq integral form}. Using the Fourier restriction norm method, we establish estimates for solutions that are uniform with respect to $\ep$ (see Propositions~\ref{prop: local in time bound for Boussinesq} and \ref{prop: Hs norm bound}). These estimates provide the key ingredient for comparing the Boussinesq and KdV flows in the low-regularity setting.

\subsection{Local-in-time bound}

We first establish local-in-time bounds for the rescaled Boussinesq system that are uniform with respect to $\ep$. The key point is that the lifespan depends only on the lower $L^2$ norm of the initial data.

\begin{proposition}[Local uniform bounds for the Boussinesq system]\label{prop: local in time bound for Boussinesq}
Let $s\geq 0$ and $\frac12<b<\frac34$, and suppose that there exist $R_0\geq 1$ and $R_s>0$ such that
$$
\sup_{\ep\in(0,1]}\|u_{\ep,0}^{\pm}\|_{L_x^2}\leq R_0,
\qquad
\sup_{\ep\in(0,1]}\|u_{\ep,0}^{\pm}\|_{H_x^s}\leq R_s.
$$
Then, there exists a time
\begin{equation}\label{eq: T choice}
T=T(R_0)\sim R_0^{-1/(\frac34-b)},
\end{equation}
independent of $\ep\in(0,1]$ and of $R_s$, and a unique solution $u_\ep^\pm\in C_t([-T,T];H_x^s)$ to the rescaled Boussinesq system \eqref{eq: coupled Boussinesq integral form} with initial data $u_{\ep,0}^{\pm}$ such that
\begin{equation}\label{eq: uniform Xsb norm bounds for Boussinesq 1}
\sup_{\ep\in(0,1]}\|u_\ep^\pm\|_{X_{\ep,\pm}^{s,b}}
\ls
R_s.
\end{equation}
\end{proposition}

\begin{remark}[Refinement in Proposition~\ref{prop: local in time bound for Boussinesq}]
\begin{enumerate}[$(i)$]

\item
A similar uniform bound was established in \cite[Proposition~5.2]{HY2024} for the KdV approximation in a low-regularity setting.

\item
Proposition~\ref{prop: local in time bound for Boussinesq} refines this previous result by showing that the lifespan in the higher-regularity space $H^s$ depends only on the lower $L^2$ norm of the initial data. More precisely, whereas the local existence time in \cite[Proposition~5.2]{HY2024} depends on $R_s$, here it depends only on the lower-regularity quantity $R_0$; see \eqref{eq: T choice}.

\item
This refinement is an instance of the \emph{persistence of regularity} property: higher $H^s$ regularity is propagated on the same lifespan as the lower-regularity solution. This property is fundamental in the study of high Sobolev norm bounds for nonlinear dispersive equations and will play a central role in the proof of Proposition~\ref{prop: Hs norm bound}; see \cite{Soh2011-1,Soh2011-2,Tao2006}.

\end{enumerate}
\end{remark}

For the proof of Proposition~\ref{prop: local in time bound for Boussinesq}, we employ the following bilinear estimates for the rescaled Boussinesq flow.

\begin{lemma}[Bilinear estimates]\label{lem: bilinear estimates}
For $s'\geq s\geq 0$ and $\frac12<b<1$, one has
\begin{equation}\label{ineq: bilin est for Boussinesq 1}
\bigg\| \frac{\pa_x}{\langle\epsilon\partial_x\rangle}(uv)\bigg\|_{X_{\ep,\pm}^{s,-\frac14}}
\ls
\|u\|_{X_{\ep,\pm}^{s,b}}
\|v\|_{X_{\ep,\pm}^{0,b}}
+
\|u\|_{X_{\ep,\pm}^{0,b}}
\|v\|_{X_{\ep,\pm}^{s,b}},
\end{equation}
\begin{equation}\label{ineq: bilin est for Boussinesq 2}
\bigg\|
\frac{\pa_x}{\langle\epsilon\partial_x\rangle}
e^{\pm\frac{2t}{\ep^2}\pa_x}(uv)
\bigg\|_{X_{\ep,\pm}^{s,-\frac14}}
\ls
\ep^{\min\{s'-s,\frac12\}}
\Big(
\|u\|_{X_{\ep,\mp}^{s',b}}
\|v\|_{X_{\ep,\mp}^{0,b}}
+
\|u\|_{X_{\ep,\mp}^{0,b}}
\|v\|_{X_{\ep,\mp}^{s',b}}
\Big),
\end{equation}
\begin{equation}\label{ineq: bilin est for Boussinesq 3}
\bigg\|
\frac{\pa_x}{\langle\epsilon\partial_x\rangle}
\Big(u\cdot e^{\pm\frac{2t}{\ep^2}\pa_x}v\Big)
\bigg\|_{X_{\ep,\pm}^{s,-\frac14}}
\ls
\ep^{\min\{s'-s,1\}}
\Big(
\|u\|_{X_{\ep,\pm}^{s',b}}
\|v\|_{X_{\ep,\mp}^{0,b}}
+
\|u\|_{X_{\ep,\pm}^{0,b}}
\|v\|_{X_{\ep,\mp}^{s',b}}
\Big).
\end{equation}
\end{lemma}

\begin{remark}[Refinement in Lemma~\ref{lem: bilinear estimates}]\label{remark: refinement of bilinear estimates}
\begin{enumerate}[$(i)$]

\item
Analogous bilinear estimates were established in \cite[Lemmas~4.3--4.5]{HY2024} and used to prove the local uniform bounds in \cite[Proposition~5.2]{HY2024}. More precisely, \eqref{ineq: bilin est for Boussinesq 1}, \eqref{ineq: bilin est for Boussinesq 2}, and \eqref{ineq: bilin est for Boussinesq 3} correspond to \cite[Lemmas~4.3, 4.4, and~4.5]{HY2024}, respectively.

\item
Lemma~\ref{lem: bilinear estimates} refines these previous estimates by reducing the total number of derivatives required on the right-hand side. More precisely, the estimate
\begin{equation}\label{eq: crude bilinear estimate}
\bigg\|
\frac{\pa_x}{\langle\epsilon\partial_x\rangle}(uv)
\bigg\|_{X_{\ep,\pm}^{s,-\frac14}}
\ls
\|u\|_{X_{\ep,\pm}^{s,b}}
\|v\|_{X_{\ep,\pm}^{s,b}},
\end{equation}
proved in \cite[Lemma~4.3]{HY2024}, is refined to \eqref{ineq: bilin est for Boussinesq 1}. This refinement is consistent with the fractional Leibniz rule.

\item
Although the proof of Lemma~\ref{lem: bilinear estimates} requires only a minor modification of the argument in \cite{HY2024} (see \eqref{ineq: fractional Leibniz rule application}), this refinement is essential for establishing the persistence of regularity.
\end{enumerate}
\end{remark}

\begin{proof}[Sketch of the proof of Lemma~\ref{lem: bilinear estimates}]

Following the standard argument of Kenig--Ponce--Vega \cite{KPV1996}, the proof of the bilinear estimates is reduced to establishing uniform bounds for certain integrals. Except for the refinement based on \eqref{ineq: fractional Leibniz rule application} (see also Remark~\ref{remark: where bilinear estimates are improved}), the proofs follow those of \cite[Lemmas~4.3, 4.4, and~4.5]{HY2024}. We therefore sketch only the reduction of \eqref{ineq: bilin est for Boussinesq 1} to the integral estimate \eqref{eq: reduction of the proof of the bilinear estimate}. By symmetry, it suffices to consider the $\|\frac{\pa_x}{\la\ep\pa_x\ra}(uv)\|_{X_{\ep,+}^{s,-\frac{1}{4}}}$ case.

To prove \eqref{ineq: bilin est for Boussinesq 1}, we apply Plancherel's theorem to obtain
$$
\bigg\|
\frac{\pa_x}{\la\ep\pa_x\ra}(uv)
\bigg\|_{X_{\ep,+}^{s,-\frac{1}{4}}}
=
\bigg\|
\frac{\la\xi\ra^s}{\la\tau+s_\ep(\xi)\ra^{\frac14}}
\frac{\xi}{\la\ep\xi\ra}
\frac{1}{(2\pi)^2}
\iint_{\R^2}
\tilde{u}(\tau_1,\xi_1)
\tilde{v}(\tau-\tau_1,\xi-\xi_1)
\, d\xi_1\, d\tau_1
\bigg\|_{L_{\tau,\xi}^2}.
$$
Applying the elementary inequality
\begin{equation}\label{ineq: fractional Leibniz rule application}
\la\xi\ra^s
\ls\la\xi_1\ra^s+\la\xi-\xi_1\ra^s,
\end{equation}
valid for $s\geq0$, we obtain
$$
\bigg\|
\frac{\pa_x}{\langle\epsilon\partial_x\rangle}(uv)
\bigg\|_{X_{\ep,+}^{s,-\frac14}}
\lesssim
\|\mathcal B(U_s,V_0)\|_{L_{\tau,\xi}^2}
+
\|\mathcal B(U_0,V_s)\|_{L_{\tau,\xi}^2},
$$
where
$$
\mathcal B(U,V)(\tau,\xi)
:=
\frac{|\xi|}
{\la\tau+s_\ep(\xi)\ra^{\frac14}\la\ep\xi\ra}
\iint_{\R^2}
\frac{U(\tau_1,\xi_1)V(\tau-\tau_1,\xi-\xi_1)}{\la\tau_1+s_\ep(\xi_1)\ra^b\la\tau-\tau_1+s_\ep(\xi-\xi_1)\ra^b}
\,d\xi_1\,d\tau_1,
$$
and
$$
U_\ell(\tau,\xi)
:=
\la\xi\ra^\ell
\la\tau+s_\ep(\xi)\ra^b
|\tilde{u}(\tau,\xi)|,
\qquad
V_\ell(\tau,\xi)
:=
\la\xi\ra^\ell
\la\tau+s_\ep(\xi)\ra^b
|\tilde{v}(\tau,\xi)|.
$$
By construction,
$$
\|U_\ell\|_{L_{\tau,\xi}^2}
=
\|u\|_{X_{\ep,+}^{\ell,b}},
\qquad
\|V_\ell\|_{L_{\tau,\xi}^2}
=
\|v\|_{X_{\ep,+}^{\ell,b}}.
$$
Hence, it suffices to prove
\begin{equation}\label{eq: reduced bilinear estimate}
\|\mathcal B(U,V)\|_{L_{\tau,\xi}^2}
\lesssim
\|U\|_{L_{\tau,\xi}^2}
\|V\|_{L_{\tau,\xi}^2}.
\end{equation}

Indeed, applying Hölder's inequality together with
$$
\int_\R
\frac{d\tau_1}
{\la\tau_1+\alpha\ra^{2b}\la\beta-\tau_1\ra^{2b}}
\ls
\frac{1}{\la\alpha+\beta\ra^{2b}},
$$
valid for $b>\frac12$, we obtain
$$
\begin{aligned}
\|\mathcal{B}(U,V)\|_{L_{\tau,\xi}^2}^2
&\leq
\Bigg\{
\sup_{(\tau,\xi)\in\R^2}
\frac{\xi^2}
{\la\tau+s_\ep(\xi)\ra^{\frac{1}{2}}\la\ep\xi\ra^2}
\iint_{\R^2}
\frac{d\xi_1\,d\tau_1}
{\la\tau_1+s_\ep(\xi_1)\ra^{2b}
\la\tau-\tau_1+s_\ep(\xi-\xi_1)\ra^{2b}}
\Bigg\}
\\
&\qquad\times
\Big\|
\|U(\tau_1,\xi_1)
V(\tau-\tau_1,\xi-\xi_1)\|_{L_{\tau_1,\xi_1}^2}
\Big\|_{L_{\tau,\xi}^2}^2
\\
&\ls
\bigg\{
\sup_{(\tau,\xi)\in\R^2}
I_\ep(\tau,\xi)
\bigg\}
\|U\|_{L_{\tau,\xi}^2}^2
\|V\|_{L_{\tau,\xi}^2}^2,
\end{aligned}
$$
where
\begin{equation}\label{eq: reduction of the proof of the bilinear estimate}
I_\ep(\tau,\xi)
:=
\frac{\xi^2}
{\la\ep\xi\ra^2
\la\tau+s_\ep(\xi)\ra^{\frac{1}{2}}}
\int_\R
\frac{d\xi_1}
{\la\tau+s_\ep(\xi_1)+s_\ep(\xi-\xi_1)\ra^{2b}}.
\end{equation}
Therefore, it remains to establish the uniform bound
$$
I_\ep(\tau,\xi)\lesssim 1,
$$
which was proved in \cite[Lemma~4.3]{HY2024}. This completes the proof of \eqref{ineq: bilin est for Boussinesq 1}. The proofs of \eqref{ineq: bilin est for Boussinesq 2} and \eqref{ineq: bilin est for Boussinesq 3} are completely analogous, relying on the corresponding integral estimates in \cite{HY2024} with the refinement based on $\la\xi\ra^s\ls\la\xi\ra^{s-s'}(\la\xi_1\ra^{s'}+\la\xi-\xi_1\ra^{s'})$ instead of \eqref{ineq: fractional Leibniz rule application}.
\end{proof}

\begin{remark}\label{remark: where bilinear estimates are improved}
In the previous work \cite{HY2024}, the crude inequality $\la\xi\ra^s\ls \la\xi_1\ra^s\la\xi-\xi_1\ra^s$ was used instead of \eqref{ineq: fractional Leibniz rule application}, which leads to \eqref{eq: crude bilinear estimate}.
\end{remark}

Now, we establish the local-in-time uniform bounds for the rescaled Boussinesq system with persistence of regularity. The proof is based on a standard contraction mapping argument, together with the completeness of the fixed point space (Lemma~\ref{lem: Completeness of the fixed point space}) and the refined bilinear estimates (Lemma~\ref{lem: bilinear estimates}).

\begin{proof}[Proof of Proposition~\ref{prop: local in time bound for Boussinesq}]
We introduce the product spaces
$$
\mathbf{X}_{\ep}^{s,b}
:=
X_{\ep,+}^{s,b}\times X_{\ep,-}^{s,b},
\qquad
\mathbf{L}^2:=L^2(\R)\times L^2(\R),
\qquad
\mathbf{H}^s:=H^s(\R)\times H^s(\R).
$$
Fix $\ep\in(0,1]$ and let
$$
\mathbf{u}=(u_\ep^+,u_\ep^-),
\qquad
\mathbf{u}_{\ep,0}
=(u_{\ep,0}^{+},u_{\ep,0}^{-})\in\mathbf{H}^s,
$$
with
$$
\|\mathbf{u}_{\ep,0}\|_{\mathbf{L}^2}\leq2R_0,
\qquad
\|\mathbf{u}_{\ep,0}\|_{\mathbf{H}^s}\leq2R_s.
$$
Let $\eta_T(t)=\eta(\frac{t}{T})$, where $\eta\in C_c^\infty(\R)$ is a smooth time cut-off function. We claim that $\mathbf{\Phi}_\ep(\mathbf{u})=(\Phi_\ep^+(\mathbf{u}),\Phi_\ep^-(\mathbf{u}))$ is a contraction mapping, where
$$
\Phi_\ep^\pm(\mathbf{u})
:=
\eta_1(t)S_\ep^\pm(t)u_{\ep,0}^\pm
\pm
\frac{\eta_1(t)}2
\int_0^t
S_\ep^\pm(t-t_1)
\eta_T(t_1)
\frac{\pa_x}{\la\ep\pa_x\ra}
\Big(
u_\ep^\pm(t_1)
+
e^{\pm\frac{2t_1}{\ep^2}\pa_x}
u_\ep^\mp(t_1)
\Big)^2
\,dt_1.
$$
Indeed, applying Lemma~\ref{lem: properties of Bourgain norm}, we obtain
$$
\begin{aligned}
\|\Phi_\ep^\pm(\mathbf{u})\|_{X_{\ep,\pm}^{s,b}}
&
\ls
\|u_{\ep,0}^{\pm}\|_{H_x^s}
+
\bigg\|
\eta_{T}
\frac{\pa_x}{\la\ep\pa_x\ra}
\Big(
u_\ep^\pm
+
e^{\pm\frac{2t}{\ep^2}\pa_x}
u_\ep^\mp
\Big)^2
\bigg\|_{X_{\ep,\pm}^{s,b-1}}
\\
&
\ls
R_s
+
T^{\frac34-b}
\bigg\|
\frac{\pa_x}{\la\ep\pa_x\ra}
\Big(
u_\ep^\pm
+
e^{\pm\frac{2t}{\ep^2}\pa_x}
u_\ep^\mp
\Big)^2
\bigg\|_{X_{\ep,\pm}^{s,-\frac14}}
\\
&=
R_s
+
T^{\frac34-b}
\bigg\|
\frac{\pa_x}{\la\ep\pa_x\ra}
\Big(
(u_\ep^\pm)^2
+
2u_\ep^\pm
e^{\pm\frac{2t}{\ep^2}\pa_x}
u_\ep^\mp
+
e^{\pm\frac{2t}{\ep^2}\pa_x}
(u_\ep^\mp)^2
\Big)
\bigg\|_{X_{\ep,\pm}^{s,-\frac14}}.
\end{aligned}
$$
Then, applying the bilinear estimates in Lemma~\ref{lem: bilinear estimates} term by term, we obtain
\begin{equation}\label{ineq: fixed point condition 1}
\|\mathbf{\Phi}_\ep(\mathbf{u})\|_{\mathbf{X}_{\ep}^{s,b}}
\le
cR_s
+
cT^{\frac34-b}
\|\mathbf{u}\|_{\mathbf{X}_{\ep}^{0,b}}
\|\mathbf{u}\|_{\mathbf{X}_{\ep}^{s,b}},
\end{equation}
for some constant $c>0$, independent of $\ep$. For the difference, similarly, one can show that 
\begin{equation}\label{ineq: fixed point condition 2}
\|
\mathbf{\Phi}_\ep(\mathbf{u}_1)
-
\mathbf{\Phi}_\ep(\mathbf{u}_2)
\|_{\mathbf{X}_{\ep}^{0,b}}
\leq
cT^{\frac34-b}
\big(
\|\mathbf{u}_1\|_{\mathbf{X}_{\ep}^{0,b}}
+
\|\mathbf{u}_2\|_{\mathbf{X}_{\ep}^{0,b}}
\big)
\|
\mathbf{u}_1-\mathbf{u}_2
\|_{\mathbf{X}_{\ep}^{0,b}}.
\end{equation}

We define
\begin{equation}\label{eq: contraction mapping space}
\mathcal{X}_c
:=
\Big\{
\mathbf{u}\in\mathbf{X}_{\ep}^{s,b}
:
\|\mathbf{u}\|_{\mathbf{X}_{\ep}^{0,b}}
\le
2cR_0,
\quad
\|\mathbf{u}\|_{\mathbf{X}_{\ep}^{s,b}}
\le
2cR_s
\Big\},
\end{equation}
equipped with the metric
$$
d(\mathbf{u},\mathbf{v})
:=
\|\mathbf{u}-\mathbf{v}\|_{\mathbf{X}_{\ep}^{0,b}}.
$$
Note that by Lemma~\ref{lem: Completeness of the fixed point space}, $(\mathcal{X}_c,d)$ is a complete metric space. Hence, choosing
$$
T
=
(8c^2R_0)^{-1/(\frac34-b)},
$$
it follows from
\eqref{ineq: fixed point condition 1}
and
\eqref{ineq: fixed point condition 2}
that
$\mathbf{\Phi}_\ep$
is a contraction on
$(\mathcal{X}_c,d)$.
The Banach fixed point theorem therefore yields a unique solution 
$u_\ep^\pm(t)$ satisfying $\|u_\ep^\pm\|_{X_{\ep,\pm}^{s,b}}\leq2cR_s$.
\end{proof}

By the same argument used to prove \eqref{ineq: bilin est for Boussinesq 1}, one obtains the bilinear estimate (see \cite{KPV1996})
\begin{equation}\label{ineq: bilin est for KdV}
\|\pa_x(uv)\|_{X_{\pm}^{s,-\frac{1}{4}}}
\ls
\|u\|_{X_{\pm}^{s,b}} \|v\|_{X_{\pm}^{0,b}}
+
\|u\|_{X_{\pm}^{0,b}} \|v\|_{X_{\pm}^{s,b}}.
\end{equation}
Repeating the proof of Proposition~\ref{prop: local in time bound for Boussinesq}, we establish the corresponding local uniform bound for the KdV equation with persistence of regularity.

\begin{proposition}[Local uniform bounds for the KdV]\label{prop: local in time bound for KdV}
Let $s\geq0$ and $\frac12<b<\frac34$, and suppose that there exist $R_0, R_s\geq1$ such that
$$
\|w_0^\pm\|_{L_x^2}
\le
R_0,
\qquad
\|w_0^\pm\|_{H_x^s}
\le
R_s.
$$
Then, there exists a time
$$
T=T(R_0)\sim R_0^{-1/(\frac34-b)},
$$
independent of $R_s$, and a unique solution $w^\pm\in C_t([-T,T];H_x^s)$ to the KdV equation \eqref{eq: KdV} with initial data $w_0^\pm$ such that
\begin{equation}\label{eq: uniform Xsb norm bounds for KdV}
\|w^\pm\|_{X_{\pm}^{s,b}} \ls R_s.
\end{equation}
\end{proposition}

\subsection{Long-time bound}

We next establish an exponential-in-time $H^s$ bound, which will be combined with the local approximation estimate in Section~\ref{sec: KdV Approximation and Global Extension} to prove the main theorem.

\begin{proposition}[Exponential bound]\label{prop: Hs norm bound}
Let $0<s\leq 5$ and $R_0,R_s>0$. There exist $\ep_0=\ep_0(R_0)>0$ and constants $C=C(R_0,s)$ and $K=K(R_0,s)$, independent of $\ep$, $R_s$, and $t$, such that the following holds. Let $\ep\in(0,\ep_0]$ and suppose that
\begin{equation}\label{ineq: R0, Rs assumption}
\sup_{\ep\in(0,\ep_0]}
\|\la\ep\pa_x\ra u_{\ep,0}^\pm\|_{L_x^2},
\,
\|w_0^\pm\|_{L_x^2}
\leq
R_0,
\qquad
\textup{and}
\qquad
\sup_{\ep\in(0,\ep_0]}
\|\la\ep\pa_x\ra u_{\ep,0}^\pm\|_{H_x^s},
\,
\|w_0^\pm\|_{H_x^s}
\leq
R_s.
\end{equation}
Let $u_\ep^\pm(t)$ (resp. $w^\pm(t)$) denote the $H^s(\R)$ solution to the rescaled Boussinesq system \eqref{eq: coupled Boussinesq integral form} (resp. the KdV equation \eqref{eq: intergal KdV}) with initial data $u_{\ep,0}^\pm$ (resp. $w_0^\pm$). Then these solutions are global and satisfy, for every $t\in\R$,
\begin{equation}\label{ineq: Hs norm bound}
\sum_\pm
\Big(
\|u_\ep^\pm(t)\|_{H_x^s}
+
\|w^\pm(t)\|_{H_x^s}
\Big)
\le
Ce^{K|t|}R_s.
\end{equation}
\end{proposition}

To iterate the local bounds in Propositions~\ref{prop: local in time bound for Boussinesq} and~\ref{prop: local in time bound for KdV}, we require uniform-in-time control of the $L^2$ norm. For the KdV equation, this follows from the $L^2$ conservation law. On the other hand, for the rescaled Boussinesq system, the rescaled conserved energy \eqref{eq: energy for rescaled Boussinesq} provides the required $L^2$ control.

\begin{lemma}\label{lem: modified L2 conservation}
For every $R>0$, there exists $\ep_0=\ep_0(R)\in(0,1]$ such that the following holds for all $\ep\in(0,\ep_0]$. Let $u_\ep^\pm$ be a solution to the rescaled Boussinesq system \eqref{eq: coupled Boussinesq integral form} with initial data $u_{\ep,0}^\pm$ satisfying
$$
\sum_\pm
\|\la\ep\pa_x\ra u_{\ep,0}^\pm\|_{L_x^2}^2
\leq
2R^2.
$$
Then, for every $t\in\R$,
\begin{equation}\label{ineq: modified L2 conservation}
\sum_\pm
\|\la\ep\pa_x\ra u_\ep^\pm(t)\|_{L_x^2}^2
\leq
3R^2.
\end{equation}
\end{lemma}

\begin{proof}
We claim that the quadratic part of the conserved energy $\mathcal{E}_\ep[u_\ep(t)]$ can be written as
\begin{equation}\label{eq: energy quadratic term}
\frac12
\big\|
\la\ep\pa_x\ra u_\ep
\big\|_{L_x^2}^2
+
\frac12
\big\|
\ep^2\pa_x^{-1}\pa_tu_\ep
\big\|_{L_x^2}^2
=
\sum_\pm
\|\la\ep\pa_x\ra u_\ep^\pm(t)\|_{L_x^2}^2.
\end{equation}
Indeed, differentiating \eqref{decomposition relation} with respect to $t$ and using \eqref{eq: coupled Boussinesq}, we obtain
$$
\begin{aligned}
\ep^2\pa_x^{-1}\partial_tu_\epsilon(t,x)
&=
\sum_\pm
\ep^2\pa_x^{-1}
\bigg\{
\partial_tu_\epsilon^\pm
\mp
\frac1{\epsilon^2}
\partial_xu_\epsilon^\pm
\bigg\}
\bigg(
t,
x\mp\frac{t}{\epsilon^2}
\bigg)
\\
&=
\sum_\pm
\bigg\{
\mp
\la\ep\partial_x\ra
u_\ep^\pm
\bigg(
t,
x\mp\frac{t}{\epsilon^2}
\bigg)
\pm
\frac{\ep^2}{2}
\frac1{\la\ep\partial_x\ra}
u_\ep^2(t,x)
\bigg\}
\\
&=
\sum_\pm
\mp
\la\ep\partial_x\ra
u_\ep^\pm
\bigg(
t,
x\mp\frac{t}{\epsilon^2}
\bigg)=
\sum_\pm
\mp
\la\ep\partial_x\ra
e^{\mp\frac{t}{\ep^2}\pa_x}
u_\ep^\pm(t,x).
\end{aligned}
$$
On the other hand, by \eqref{decomposition relation}, we have
$$
\la\ep\pa_x\ra u_\ep
=
\sum_\pm
\la\ep\partial_x\ra
u_\ep^\pm
\bigg(
t,
x\mp\frac{t}{\epsilon^2}
\bigg)
=
\sum_\pm
\la\ep\partial_x\ra
e^{\mp\frac{t}{\ep^2}\pa_x}
u_\ep^\pm(t,x).
$$
Therefore, \eqref{eq: energy quadratic term} follows from the identity
$\|a+b\|_{L^2}^2+\|a-b\|_{L^2}^2=2\|a\|_{L^2}^2
+2\|b\|_{L^2}^2.$

By the claim \eqref{eq: energy quadratic term} and the conservation of energy, we obtain
$$
\begin{aligned}
\sum_\pm
\|\la\ep\pa_x\ra u_\ep^\pm(t)\|_{L_x^2}^2
&=
\mathcal{E}_\ep[u_\ep(t)]
+
\frac{\ep^2}{3}
\int_\R
u_\ep^3(t)\,dx=
\mathcal{E}_\ep[u_\ep(0)]
+
\frac{\ep^2}{3}
\int_\R
u_\ep^3(t)\,dx
\\
&=
\sum_\pm
\|\la\ep\pa_x\ra u_{\ep,0}^\pm\|_{L_x^2}^2
-
\frac{\ep^2}{3}
\int_\R
u_\ep^3(0)\,dx
+
\frac{\ep^2}{3}
\int_\R
u_\ep^3(t)\,dx
\\
&\le
2R^2
+
\frac{\ep^2}{3}
\Big(
\|u_\ep(t)\|_{L_x^3}^3
+
\|u_\ep(0)\|_{L_x^3}^3
\Big),
\end{aligned}
$$
where we used the assumption on the initial data. Next, to estimate the cubic terms, we apply the 1D Gagliardo--Nirenberg inequality:
$$
\begin{aligned}
\ep^2\|u_\ep\|_{L_x^3}^3
&\lesssim
\sum_\pm\ep^2\|u_\ep^\pm\|_{L_x^3}^3
\lesssim
\sum_\pm
\ep^{\frac32}\|u_\ep^\pm\|_{L_x^2}^{\frac52}
\Big(\ep\|\pa_xu_\ep^\pm\|_{L_x^2}\Big)^{\frac12}
\lesssim
\ep^{\frac32}
\Biggl(
\sum_\pm
\|\la\ep\pa_x\ra
u_\ep^\pm\|_{L_x^2}^2
\Biggr)^{\frac32}.
\end{aligned}
$$
Therefore, combining the above estimate with the previous inequality gives
$$
\sum_\pm
\|\la\ep\pa_x\ra u_\ep^\pm(t)\|_{L_x^2}^2
\le
2R^2
+
C\ep^{\frac32}
\Biggl(
\sum_\pm
\|\la\ep\pa_x\ra u_\ep^\pm(t)\|_{L_x^2}^2
\Biggr)^{\frac32}
+
C\ep^{\frac32}
2\sqrt{2}R^3.
$$

Assume that
$$
\sum_\pm
\|\la\ep\pa_x\ra u_\ep^\pm(t)\|_{L_x^2}^2
\leq4R^2.
$$
Then, we have
$$
\sum_\pm
\|\la\ep\pa_x\ra u_\ep^\pm(t)\|_{L_x^2}^2
\leq
2R^2+(8+2\sqrt2)C\ep^{\frac32}R^3
\leq3R^2,
$$
provided that $\ep>0$ is sufficiently small, depending only on $R$. Therefore, by continuity of solution in time, a standard bootstrap argument yields \eqref{ineq: modified L2 conservation}.
\end{proof}

\begin{proof}[Proof of Proposition~\ref{prop: Hs norm bound}]
Let $\ep_0=\ep_0(R_0)$ be the constant given by Lemma~\ref{lem: modified L2 conservation}. Then, for every $t\in\R$, the $L^2$ conservation law for the KdV equation together with Lemma~\ref{lem: modified L2 conservation} yields
\begin{equation}\label{ineq: uniform L2 bound}
\|w^\pm(t)\|_{L_x^2}
=
\|w_0^\pm\|_{L_x^2}
\le
R_0,
\qquad
\|u_\ep^\pm(t)\|_{L_x^2}^2
\le
\sum_\pm
\|\la\ep\pa_x\ra u_\ep^\pm(t)\|_{L_x^2}^2
\le
3R_0^2.
\end{equation}
Hence, the $L^2$ norms of both solutions remain uniformly bounded by $\sqrt3R_0$. On the other hand, applying Propositions~\ref{prop: local in time bound for Boussinesq} and~\ref{prop: local in time bound for KdV}, there exist
$$
\Delta t=\Delta t(R_0)>0,
\qquad
C=C(R_0,s)\ge1,
$$
independent of $\ep$ and $R_s$, such that for every $t\in\R$,
\begin{equation}\label{ineq: iteration step}
\sum_\pm
\|u_\ep^\pm(t+\Delta t)\|_{H_x^s}
\le
C
\sum_\pm
\|u_\ep^\pm(t)\|_{H_x^s},
\qquad
\sum_\pm
\|w^\pm(t+\Delta t)\|_{H_x^s}
\le
C
\sum_\pm
\|w^\pm(t)\|_{H_x^s}.
\end{equation}
Therefore, iterating \eqref{ineq: iteration step} yields
$$
\sum_\pm
\Big(
\|u_\ep^\pm(t)\|_{H_x^s}
+
\|w^\pm(t)\|_{H_x^s}
\Big)
\leq
Ce^{K|t|}R_s,
$$
for some constant $K=K(R_0,s)>0$, which proves the proposition.
\end{proof}

\section{KdV Approximation and Global Extension}\label{sec: KdV Approximation and Global Extension}

In this section, we complete the proof of the main theorem (Theorem~\ref{thm: main theorem}). We first refine the local-in-time approximation result of Hong and Yang \cite{HY2024} into a form suitable for the iteration argument. More precisely, by combining the persistence of regularity argument with Propositions~\ref{prop: local in time bound for Boussinesq} and~\ref{prop: local in time bound for KdV}, we obtain a local approximation estimate whose lifespan depends only on the lower $L^2$ norm. We then iterate this local result to complete the proof.

\subsection{Frequency-localized decomposition}\label{subsec: Frequency localized decomposition}
To prove the local KdV approximation without imposing additional regularity on the initial data, we use the frequency-localized decomposition introduced by Hong and Yang \cite{HY2024}, which avoids a direct high-frequency comparison between the Boussinesq and Airy flows.

As mentioned in Remark~\ref{rmk: Sensitivity of the Bourgain norm}, the direct comparison of two solutions $u_\ep^\pm$ and $w^\pm$ in a single Bourgain norm may require additional regularity. Indeed, for $|\xi|\gg \ep^{-1}$,
$$
s_\ep(\xi)\sim \frac{|\xi|\xi}{\ep}\quad \textup{and}\quad s(\xi)=\frac{\xi^3}{2},
$$
so the rescaled Boussinesq and Airy flows have completely different high-frequency behavior.

To overcome this difficulty, Hong and Yang introduced the following \textit{frequency-localized decoupled Boussinesq equation}, which we use here without modification:
\begin{equation}\label{eq: auxiliary equation}
\boxed{\quad v_{\ep}^{\pm}(t)
=S_\ep^{\pm}(t)P_{\leq N_0}u_{\epsilon,0}^{\pm} \pm \frac{1}{2}\int_0^tS_\ep^{\pm}(t-t_1)  P_{\leq N_0}\frac{\pa_x }{\la\ep \pa_x\ra}(P_{\leq N_0}v_\ep^\pm (t_1))^2   dt_1,\quad}
\end{equation}
where 
$$
N_0=\frac{1}{2}\ep^{-\frac{2}{5}},
$$
and $P_{\leq N_0}$ is the sharp frequency projection defined by $\widehat{P_{\leq N_0}f}(\xi)=\mathbf{1}_{\{|\xi|\leq N_0\}}\hat{f}(\xi)$.

The cut-off frequency $N_0$ is chosen so that the Boussinesq and Airy phases are comparable on $|\xi|\leq N_0$, while the complementary high-frequency part can be controlled by the $H^s$-regularity. Indeed, Taylor expansion gives
\begin{equation}\label{ineq: phase difference}
|s_\ep(\xi)-s(\xi)|
\ls
\ep^2|\xi|^5
\ls
\ep^{\frac{2s}{5}}|\xi|^s
\quad
\textup{where}
\quad
|\xi|\leq N_0
\quad \textup{and}\quad 0\leq s\leq5.
\end{equation}
The same choice of $N_0$ yields
$$
\|(1-P_{\leq N_0})f\|_{L_x^2}
\ls
N_0^{-s}\|f\|_{H_x^s}
\ls
\ep^{\frac{2s}{5}}\|f\|_{H_x^s}.
$$
Thus, the low-frequency phase error and high-frequency truncation both contribute the same order $\ep^{\frac{2s}{5}}$. The following estimates of Hong and Yang quantify low-frequency comparison.
\begin{lemma}[Low-frequency comparison estimates \cite{HY2024}]\label{lem: Low frequency comparison estimates}
Let $N_0=\frac{1}{2}\ep^{-\frac{2}{5}}$ and $P_{\leq N_0}$ be the sharp frequency cut-off. Set $\eta_T(t)=\eta(t/T)$ where $\eta\in C_c^\infty(\R)$ is the smooth time cut-off defined in \eqref{eq: smooth cutoff}. Then, for $s,b\in \R$ and $0<T\leq1$, one has
\begin{equation}\label{ineq: low freq est 1}
\|P_{\leq N_0}f\|_{X_{\pm}^{s,b}}
\sim 
\|P_{\leq N_0}f\|_{X_{\ep,\pm}^{s,b}}.
\end{equation}
If we further assume $0\leq s\leq 5$ and $\frac{1}{2}<b\leq 1$, we have
\begin{equation}\label{ineq: low freq est 2}
\|\eta_T(t)(S_\ep^\pm(t)-S^\pm(t))P_{\leq N_0}u_0\|_{X_{\pm}^{0,b}}
\ls
\ep^{\frac{2s}{5}}T^{\frac{3}{2}-b}\|u_0\|_{H_x^s},
\end{equation}
and
\begin{equation}\label{ineq: low freq est 3}
\bigg\|
\eta_T(t)\int_0^t\big(S_\ep^\pm(t-t')-S^\pm(t-t')\big)\eta_T(t')(P_{\leq N_0}F)(t')\, dt'
\bigg\|_{X_{\pm}^{0,b}}
\ls
\ep^{\frac{2s}{5}}T^{\frac{3}{2}-b}\|F\|_{X_{\pm}^{s,b-1}}.
\end{equation}
\end{lemma}
Since $v_\ep^\pm$ is supported in $|\xi|\leq N_0$, \eqref{ineq: low freq est 1} gives
\begin{equation}
\|v_\ep^\pm\|_{X_{\ep,\pm}^{s,b}}
\sim
\|v_\ep^\pm\|_{X_{\pm}^{s,b}}.
\end{equation}
Moreover, arguing as in the proof of Proposition~\ref{prop: local in time bound for Boussinesq}, we obtain the following local uniform bound with persistence of regularity for the auxiliary equation \eqref{eq: auxiliary equation}. This is because the auxiliary equation has the same structure as the rescaled Boussinesq system \eqref{eq: coupled Boussinesq integral form} but fewer nonlinear terms, and the frequency cut-off $P_{\leq N_0}$ only makes the terms smaller. 
\begin{proposition}[Local uniform bounds for the auxiliary equation]\label{prop: local in time bound for auxiliary equation}
Let $0\leq s\leq 5$ and $\frac{1}{2}<b<\frac{3}{4}$. Suppose that there exist $R_0, R_s\geq1$ such that
$$
\sup_{\ep\in(0,1]}\|u_{\ep,0}^{\pm}\|_{L_x^2}\leq R_0,
\qquad
\sup_{\ep\in(0,1]}\|u_{\ep,0}^{\pm}\|_{H_x^s}\leq R_s.
$$
Then, there exists a time $T=T(R_0)\sim R_0^{-1/(\frac{3}{4}-b)}$, independent of $\ep\in(0,1]$ and $R_s$, and a unique solution $v_\ep^\pm\in C_t([-T,T];H_x^s)$ to the frequency-localized decoupled Boussinesq equation \eqref{eq: auxiliary equation} with initial data $P_{\leq N_0}u_{\ep,0}^{\pm}$ such that
\begin{equation}\label{eq: uniform Xsb norm bounds for auxiliary}
\sup_{\ep\in(0,1]}\|v_\ep^\pm\|_{X_{\pm}^{s,b}}
\sim
\sup_{\ep\in(0,1]}\|v_\ep^\pm\|_{X_{\ep,\pm}^{s,b}}
\ls
R_s.
\end{equation}
\end{proposition}
As in \cite{HY2024}, with the auxiliary solution $v_\ep^\pm$ constructed above, we estimate $u_\ep^\pm-v_\ep^\pm$ in $X_{\ep,\pm}^{0,b}$ and $v_\ep^\pm-w^\pm$ in $X_\pm^{0,b}$ separately, instead of comparing $u_\ep^\pm$ and $w^\pm$ directly. This approach allows us to establish the local justification without imposing additional regularity requirement on the initial data.

\subsection{Local approximation with persistence of regularity}

We now refine the local KdV approximation of Hong and Yang \cite{HY2024} into a form suitable for the iteration argument. As in the local theory developed in Section~\ref{sec: Uniform bounds for the Boussinesq system}, the lifespan of the approximation depends only on the lower $L^2$ norm, while the higher $H^s$ norm enters only through the approximation error. Since the $L^2$ norm is conserved for the KdV equation and remains uniformly bounded for the rescaled Boussinesq system by Lemma~\ref{lem: modified L2 conservation}, this local approximation can be iterated.

\begin{proposition}[Local KdV approximation with persistence of regularity]\label{prop: KdV to Boussinesq local result}
Let $0\le s\le5$ and suppose that
$$
\sup_{\ep\in(0,1]}
\|u_{\ep,0}^\pm\|_{L_x^2},\,
\|w_0^\pm\|_{L_x^2}
\leq
R_0,
\qquad
\textup{and}
\qquad
\sup_{\ep\in(0,1]}
\|u_{\ep,0}^\pm\|_{H_x^s},
\,
\|w_0^\pm\|_{H_x^s}
\leq
R_s.
$$
For each $\ep\in(0,1]$, let $u_\ep^\pm(t)$ (resp. $w^\pm(t)$) denote the unique $H^s$-solution to the rescaled Boussinesq system \eqref{eq: coupled Boussinesq integral form}
(resp. the KdV equation \eqref{eq: intergal KdV}) with initial data $u_{\ep,0}^\pm$ (resp. $w_0^\pm$). Then, there exists a time $T=T(R_0)>0$ and a constant $C_s=C_s(R_0)>1$, both independent of $\ep$ and $R_s$, such that
\begin{equation}\label{ineq: Boussinesq to KdV local difference}
\|u_\ep^\pm-w^\pm\|_{C_t([-T,T];L_x^2)}
\leq
C_s\|u_{\ep,0}^\pm-w_0^\pm\|_{L_x^2}
+
C_sR_s\ep^{\min\{\frac{2s}{5},\frac12\}}.
\end{equation}
\end{proposition}

\begin{proof}
Let $\frac12<b<\frac34$. By Propositions~\ref{prop: local in time bound for Boussinesq}, \ref{prop: local in time bound for KdV} and \ref{prop: local in time bound for auxiliary equation}, the solutions $u_\ep^\pm(t)$, $v_\ep^\pm(t)$ and $w^\pm(t)$ exist on the common time interval $[-T,T]$, where $v_\ep^\pm(t)$ is defined by \eqref{eq: auxiliary equation} and
\begin{equation}\label{eq: choice of T for local KdV limit}
T
:=
\min\left\{
\left(\frac{c_0}{R_0}\right)^{1/(\frac34-b)},
\,\frac{1}{2}
\right\},
\end{equation}
for a sufficiently small constant $c_0>0$ to be chosen later. Moreover, they satisfy
\begin{equation}\label{ineq: uniform bound for local KdV limit}
\|u_\ep^\pm\|_{X_{\ep,\pm}^{0,b}},
\,
\|v_\ep^\pm\|_{X_{\ep,\pm}^{0,b}},
\,
\|w^\pm\|_{X_{\pm}^{0,b}}
\ls
R_0,
\qquad
\textup{and}
\qquad
\|u_\ep^\pm\|_{X_{\ep,\pm}^{s,b}},
\,
\|v_\ep^\pm\|_{X_{\ep,\pm}^{s,b}},
\,
\|w^\pm\|_{X_{\pm}^{s,b}}
\ls
R_s.
\end{equation}

We estimate the differences $u_\ep^\pm-v_\ep^\pm$ and $v_\ep^\pm-w^\pm$ separately. We first consider $u_\ep^\pm-v_\ep^\pm$. Inserting smooth time cut-offs $\eta_1$ and $\eta_T$ defined in \eqref{eq: smooth cutoff}, we write
$$
\begin{aligned}
u_\ep^\pm(t)-v_\ep^\pm(t)
&=
\eta_1(t)S_\ep^\pm(t)(1-P_{\leq N_0}) u_{\ep,0}^\pm\\
&\quad\pm
\frac{\eta_1(t)}{2}
\int_0^tS_\ep^\pm(t-t')\eta_T(t')
\frac{\pa_x}{\la\ep\pa_x\ra}
\big\{(u_\ep^\pm)^2-(v_\ep^\pm)^2\big\}(t')\, dt'\\
&\quad\pm
\frac{\eta_1(t)}{2}
\int_0^tS_\ep^\pm(t-t')\eta_T(t')
\frac{\pa_x}{\la\ep\pa_x\ra}
(1-P_{\leq N_0})(v_\ep^\pm)^2(t')\, dt'\\
&\quad
\pm
\frac{\eta_1(t)}{2}
\int_0^tS_\ep^\pm(t-t')\eta_T(t')
\frac{\pa_x}{\la\ep\pa_x\ra}
\big\{2u_\ep^\pm e^{\pm\frac{2t'}{\ep^2}\pa_x}u_\ep^\mp
+e^{\pm\frac{2t'}{\ep^2}\pa_x}(u_\ep^\mp)^2\big\}(t')\, dt'.
\end{aligned}
$$
The linear estimates in Lemma~\ref{lem: properties of Bourgain norm} yield
$$
\begin{aligned}
&\|u_\ep^\pm-v_\ep^\pm\|_{X_{\ep,\pm}^{0,b}}\\
&\ls
\|(1-P_{\leq N_0})u_{\ep,0}^\pm\|_{L_x^2}
+
\bigg\|
\eta_T\frac{\pa_x}{\la\ep\pa_x\ra}
\big\{(u_\ep^\pm+v_\ep^\pm)(u_\ep^\pm-v_\ep^\pm)\big\}
\bigg\|_{X_{\ep,\pm}^{0,b-1}} \\
&\quad
+
\bigg\|\eta_T\frac{\pa_x}{\la\ep\pa_x\ra}
(1-P_{\leq N_0})(v_\ep^\pm)^2\bigg\|_{X_{\ep,\pm}^{0,b-1}}
+
\bigg\|
\eta_T\frac{\pa_x}{\la\ep\pa_x\ra}
\big\{2u_\ep^\pm e^{\pm\frac{2t}{\ep^2}\pa_x}u_\ep^\mp
+e^{\pm\frac{2t}{\ep^2}\pa_x}(u_\ep^\mp)^2\big\}
\bigg\|_{X_{\ep,\pm}^{0,b-1}} \\
&\ls
N_0^{-s}\|u_{\ep,0}^\pm\|_{H_x^s}
+T^{\frac34-b}
\bigg\|
\frac{\pa_x}{\la\ep\pa_x\ra}
\big\{(u_\ep^\pm+v_\ep^\pm)(u_\ep^\pm-v_\ep^\pm)\big\}
\bigg\|_{X_{\ep,\pm}^{0,-\frac14}}
+T^{\frac34-b}N_0^{-s}
\bigg\|
\frac{\pa_x}{\la\ep\pa_x\ra}
(v_\ep^\pm)^2
\bigg\|_{X_{\ep,\pm}^{s,-\frac14}}
\\
&\quad
+T^{\frac34-b}
\bigg\|
\frac{\pa_x}{\la\ep\pa_x\ra}
\big\{u_\ep^\pm (e^{\pm\frac{2t}{\ep^2}\pa_x}u_\ep^\mp)\big\}
\bigg\|_{X_{\ep,\pm}^{0,-\frac14}}
+T^{\frac34-b}
\bigg\|
\frac{\pa_x}{\la\ep\pa_x\ra}
\big\{e^{\pm\frac{2t}{\ep^2}\pa_x}(u_\ep^\mp)^2\big\}
\bigg\|_{X_{\ep,\pm}^{0,-\frac14}}.
\end{aligned}
$$
Then, applying the bilinear estimates in Lemma~\ref{lem: bilinear estimates} and using the uniform bounds \eqref{ineq: uniform bound for local KdV limit}, we obtain
$$
\begin{aligned}
\|u_\ep^\pm-v_\ep^\pm\|_{X_{\ep,\pm}^{0,b}}
&\ls
N_0^{-s}\|u_{\ep,0}^\pm\|_{H_x^s}
+T^{\frac34-b}
\big(
\|u_\ep^\pm\|_{X_{\ep,\pm}^{0,b}}
+\|v_\ep^\pm\|_{X_{\ep,\pm}^{0,b}}
\big)
\|u_\ep^\pm-v_\ep^\pm\|_{X_{\ep,\pm}^{0,b}} \\
&\quad 
+T^{\frac34-b}N_0^{-s}\|v_\ep^\pm\|_{X_{\ep,\pm}^{s,b}}\|v_\ep^\pm\|_{X_{\ep,\pm}^{0,b}}\\
&\quad
+T^{\frac34-b}\ep^{\min\{s,1\}}
\big(
\|u_\ep^\pm\|_{X_{\ep,\pm}^{s,b}}
\|u_\ep^\mp\|_{X_{\ep,\mp}^{0,b}}
+
\|u_\ep^\pm\|_{X_{\ep,\pm}^{0,b}}
\|u_\ep^\mp\|_{X_{\ep,\mp}^{s,b}}
\big)\\
&\quad
+T^{\frac34-b}\ep^{\min\{s,\frac12\}}
\|u_\ep^\mp\|_{X_{\ep,\mp}^{s,b}}
\|u_\ep^\mp\|_{X_{\ep,\mp}^{0,b}}
\\
&\ls
\ep^{\frac{2s}{5}}R_s
+
c_0\|u_\ep^\pm-v_\ep^\pm\|_{X_{\ep,\pm}^{0,b}}
+
c_0\ep^{\frac{2s}{5}}R_s
+
c_0\ep^{\min\{s,1\}}R_s
+
c_0\ep^{\min\{s,\frac12\}}R_s.
\end{aligned}
$$
Here, we used the choice of $T$ in \eqref{eq: choice of T for local KdV limit} and $N_0=\frac12\ep^{-\frac25}$. Therefore, choosing $c_0>0$ sufficiently small (but independently of $\epsilon\in(0,1]$) and absorbing the second term on the right-hand side, we obtain
\begin{equation}\label{ineq: u-v estimate}
\|u_\ep^\pm-v_\ep^\pm\|_{X_{\ep,\pm}^{0,b}}
\ls
\ep^{\min\{\frac{2s}{5},\frac12\}}R_s.
\end{equation}

Next, we shall estimate $v_\ep^\pm-w^\pm$. As before, we again insert smooth temporal cut-offs.
$$
\begin{aligned}
\|v_\ep^\pm-w^\pm\|_{X_{\pm}^{0,b}}
&
\leq
\|\eta_1(1-P_{\leq N_0})w^\pm\|_{X_\pm^{0,b}}
+
\|\eta_1(S_\ep^\pm(t)-S^\pm(t))P_{\leq N_0}u_{\ep,0}^\pm\|_{X_\pm^{0,b}}\\
& \quad +
\|\eta_1S^\pm(t) P_{\leq N_0}(u_{\ep,0}^\pm-w_0^\pm)\|_{X_\pm^{0,b}} \\
&\quad +
\bigg\|
\frac{\eta_1}{2}\int_0^t
S_\ep^\pm(t-t')\eta_T(t')\frac{\pa_x}{\la\ep\pa_x\ra}
P_{\leq N_0}\{(v_\ep^\pm)^2-(w^\pm)^2\}(t') \, dt'
\bigg\|_{X_\pm^{0,b}} \\
&\quad +
\bigg\|
\frac{\eta_1}{2}\int_0^t
(S_\ep^\pm(t-t')-S^\pm(t-t'))\frac{\pa_x}{\la\ep\pa_x\ra}(\eta_1\eta_T)(t')P_{\leq N_0}\{(w^\pm)^2(t')\}\, dt'
\bigg\|_{X_\pm^{0,b}} \\
&\quad +
\bigg\|
\frac{\eta_1}{2}\int_0^t
S^\pm(t-t')\eta_T(t')\bigg(1-\frac{1}{\la\ep\pa_x\ra}\bigg)\pa_xP_{\leq N_0}\{(w^\pm)^2(t')\}\, dt'
\bigg\|_{X_\pm^{0,b}} \\
&
=:\textup{(I)} + \textup{(II)} + \textup{(III)} + \textup{(IV)} + \textup{(V)} + \textup{(VI)}.
\end{aligned}
$$
We group the terms according to analogous estimates and treat each group together. By the linear estimates in Lemma~\ref{lem: properties of Bourgain norm} and uniform bounds \eqref{ineq: uniform bound for local KdV limit} with $N_0=\frac{1}{2}\ep^{-\frac{2}{5}}$, we obtain
\begin{equation}\label{ineq: local KdV limit estimate 1}
\begin{aligned}
\textup{(I)}+\textup{(III)}
&\ls
\|(1-P_{\leq N_0})w^\pm\|_{X_\pm^{0,b}}
+
\|P_{\leq N_0}(u_{\ep,0}^\pm-w_0^\pm)\|_{L_x^2}\\
&\ls
N_0^{-s}\|w^\pm\|_{X_\pm^{s,b}}+\|u_{\ep,0}^\pm-w_0^\pm\|_{L_x^2}
\ls
\ep^{\frac{2s}{5}}R_s+\|u_{\ep,0}^\pm-w_0^\pm\|_{L_x^2}.
\end{aligned}
\end{equation}
On the other hand, applying Lemma~\ref{lem: Low frequency comparison estimates}, linear estimates in Lemma~\ref{lem: properties of Bourgain norm}, and the choice of $T$ in \eqref{eq: choice of T for local KdV limit} with uniform bounds \eqref{ineq: uniform bound for local KdV limit}, we have
\begin{equation}\label{ineq: local KdV limit estimate 2}
\begin{aligned}
\textup{(II)}+\textup{(V)}
&\ls
\ep^{\frac{2s}{5}}\|u_{\ep,0}^\pm\|_{H_x^s}
+
\ep^{\frac{2s}{5}}\bigg\|\eta_T\frac{\pa_x}{\la\ep\pa_x\ra}(w^\pm)^2\bigg\|_{X_\pm^{s,b-1}}\\
&\ls 
\ep^{\frac{2s}{5}}\|u_{\ep,0}^\pm\|_{H_x^s}
+
\ep^{\frac{2s}{5}}T^{\frac{3}{4}-b}\|\pa_x(w^\pm)^2\|_{X_\pm^{s,-\frac{1}{4}}}\\
&
\ls
\ep^{\frac{2s}{5}}\|u_{\ep,0}^\pm\|_{H_x^s}
+
\ep^{\frac{2s}{5}}T^{\frac{3}{4}-b}\|w^\pm\|_{X_\pm^{0,b}}\|w^\pm\|_{X_\pm^{s,b}}
\ls \ep^{\frac{2s}{5}}R_s+\ep^{\frac{2s}{5}}c_0R_s.
\end{aligned}
\end{equation}
For the remaining terms, using \eqref{ineq: low freq est 1} and the inhomogeneous estimate in Lemma~\ref{lem: properties of Bourgain norm} with the fact that $1-\frac{1}{\la\ep\xi\ra}=\frac{\ep^2\xi^2}{\la\ep\xi\ra(1+\la\ep\xi\ra)}\ls (\ep|\xi|)^{\frac{2s}{5}}\ls \ep^{\frac{2s}{5}}\la\xi\ra^s$ when $|\xi|\leq N_0=\frac{1}{2}\ep^{-\frac{2}{5}}$, we get
$$
\textup{(IV)}+\textup{(VI)}
\ls
\bigg\|
\eta_T\frac{\pa_x}{\la\ep\pa_x\ra}\{(v_\ep^\pm)^2-(w^\pm)^2\}
\bigg\|_{X_\pm^{0,b-1}}
+
\ep^{\frac{2s}{5}}
\|\eta_T\pa_x\{(w^\pm)^2\}\|_{X_\pm^{s,b-1}}.
$$
Then, the linear estimates in Lemma~\ref{lem: properties of Bourgain norm} and the bilinear estimate for KdV \eqref{ineq: bilin est for KdV} imply
\begin{equation}\label{ineq: local KdV limit estimate 3}
\begin{aligned}
\textup{(IV)}+\textup{(VI)}
&\ls
\bigg\|
\eta_T\frac{\pa_x}{\la\ep\pa_x\ra}\{(v_\ep^\pm)^2-(w^\pm)^2\}
\bigg\|_{X_\pm^{0,b-1}}
+
\ep^{\frac{2s}{5}}
\|\eta_T\pa_x\{(w^\pm)^2\}\|_{X_\pm^{s,b-1}} \\
&
\ls
T^{\frac{3}{4}-b}\|\pa_x\{(v_\ep^\pm)^2-(w^\pm)^2\}\|_{X_\pm^{0,-\frac{1}{4}}}
+
\ep^{\frac{2s}{5}}T^{\frac{3}{4}-b}\|\pa_x\{(w^\pm)^2\}\|_{X_\pm^{s,-\frac{1}{4}}}\\
&
\ls
T^{\frac{3}{4}-b}
\big(
\|v_\ep^\pm\|_{X_\pm^{0,b}}+\|w^\pm\|_{X_\pm^{0,b}}
\big)
\|v_\ep^\pm-w^\pm\|_{X_\pm^{0,b}}
+
\ep^{\frac{2s}{5}}T^{\frac{3}{4}-b}\|w^\pm\|_{X_\pm^{0,b}}\|w^\pm\|_{X_\pm^{s,b}}\\
&
\ls c_0\|v_\ep^\pm-w^\pm\|_{X_\pm^{0,b}}+\ep^{\frac{2s}{5}}R_s,
\end{aligned}
\end{equation}
where in the last inequality, we used uniform bounds \eqref{ineq: uniform bound for local KdV limit} and the definition of $T$ in \eqref{eq: choice of T for local KdV limit}. Hence, combining \eqref{ineq: local KdV limit estimate 1}--\eqref{ineq: local KdV limit estimate 3}, we conclude that
\begin{equation}\label{ineq: v-w estimate}
\|v_\ep^\pm-w^\pm\|_{X_{\pm}^{0,b}}
\ls \|u_{\ep,0}^\pm-w_0^\pm\|_{L_x^2} + \ep^{\frac{2s}{5}}R_s,
\end{equation}
after choosing $c_0>0$ sufficiently small. The proposition follows from \eqref{ineq: u-v estimate} and \eqref{ineq: v-w estimate}.
\end{proof}

\subsection{Global extension}\label{sec: Global Extension}

In this section, we complete the proof of the main theorem by iterating the local KdV approximation established in Proposition~\ref{prop: KdV to Boussinesq local result}. The key point is that the lifespan in Proposition~\ref{prop: KdV to Boussinesq local result} depends only on the lower $L^2$ norm, rather than the higher $H^s$ norm. Since the $L^2$ norm is conserved for the KdV equation and remains uniformly bounded for the rescaled Boussinesq system by Lemma~\ref{lem: modified L2 conservation}, the local approximation can be iterated with a uniform time step. Combining Proposition~\ref{prop: Hs norm bound} with Proposition~\ref{prop: KdV to Boussinesq local result}, we obtain the main theorem.

\begin{proof}[Proof of Theorem~\ref{thm: main theorem}]
It suffices to consider the case $t\ge0$, since the argument for $t\le0$ is identical. Throughout the proof, we fix $R_0,R_s>0$ satisfying \eqref{ineq: R0, Rs assumption}. By Proposition~\ref{prop: Hs norm bound}, for all sufficiently small $\ep\in(0,1]$, the exponential bound \eqref{ineq: Hs norm bound} holds for all $t\ge0$.

Applying Proposition~\ref{prop: KdV to Boussinesq local result}, the local approximation can be iterated with a time step
$$
\Delta t=\Delta t(R_0)>0,
$$
independent of $\ep$, $R_s$, and the initial time. Enlarging the constant if necessary, there exists $C=C(R_0,s)\geq2$ such that for every $\tau\geq0$,
$$
\|u_\ep^\pm-w^\pm\|_{C_t([\tau,\tau+\Delta t];L_x^2)}
\leq
C\|u_\ep^\pm(\tau)-w^\pm(\tau)\|_{L_x^2}
+
Ce^{C\tau}R_s\ep^{\min\{\frac{2s}{5},\frac12\}}.
$$
Let $t_n:=n\Delta t$, $n\in\mathbb Z_{\geq0}$.
Applying the previous estimate with $\tau=t_n$, we obtain
\begin{equation}\label{ineq: one step iteration}
\|u_\ep^\pm-w^\pm\|_{C_t([t_n,t_{n+1}];L_x^2)}
\leq
C\|u_\ep^\pm(t_n)-w^\pm(t_n)\|_{L_x^2}
+
Ce^{Ct_n}R_s
\ep^{\min\{\frac{2s}{5},\frac12\}}.
\end{equation}
Then, iterating \eqref{ineq: one step iteration}, we obtain, for every integer $m\geq 1$,
$$
\|u_\ep^\pm-w^\pm\|_{C_t([0,t_m];L_x^2)}
\leq
C^m
\|u_{\ep,0}^\pm-w_0^\pm\|_{L_x^2}
+
CR_s
\ep^{\min\{\frac{2s}{5},\frac12\}}
\sum_{k=0}^{m-1}
C^{m-k}
e^{Ct_k}.
$$
Choose $K=K(R_0,s)>0$ sufficiently large so that
$$
C(C+e^{C\Delta t})
\leq
e^{K\Delta t}.
$$
Since $t_k=k\Delta t$, we have $C^m\leq e^{Kt_m}$ and
$$
C\sum_{k=0}^{m-1}
C^{m-k}e^{Ct_k}
=
C\sum_{k=0}^{m-1}C^{m-k}
(e^{C\Delta t})^k
\leq
C(C+e^{C\Delta t})^m
\leq
e^{Kt_m}.
$$
Therefore,
$$
\|u_\ep^\pm-w^\pm\|_{C_t([0,t_m];L_x^2)}
\leq
e^{Kt_m}
\left(
\|u_{\ep,0}^\pm-w_0^\pm\|_{L_x^2}
+
R_s
\ep^{\min\{\frac{2s}{5},\frac12\}}
\right).
$$
Given $t\geq0$, choose $m\in\Z_{\geq 1}$ such that $t\leq t_m\leq t+\Delta t$. It follows that
$$
\|u_\ep^\pm(t)-w^\pm(t)\|_{L_x^2}
\leq
e^{Kt}
e^{K\Delta t}
\left(
\|u_{\ep,0}^\pm-w_0^\pm\|_{L_x^2}
+
R_s
\ep^{\min\{\frac{2s}{5},\frac12\}}
\right).
$$
Finally, set
$$
K_0:=\max\{K, e^{K\Delta t}\}, \qquad K_s:=K_0R_s.
$$
By construction, $K_0$ depends only on $R_0$ and $s$, whereas $K_s$ may additionally depend on $R_s$. Hence, the preceding estimate yields \eqref{eq: convergence bound}, as we desired.
\end{proof}

\bibliographystyle{abbrv}
\bibliography{Reference}

\end{document}